\documentclass{amsart}

\newtheorem{theo}{Theorem}[section]
\newtheorem{prop}[theo]{Proposition}
\newtheorem{lemm}[theo]{Lemma}
\newtheorem{coro}[theo]{Corollary}
\newtheorem{defi}[theo]{Definition}
\newtheorem{rema}[theo]{Remark}

\newtheorem{thmnonum}{Theorem}

\DeclareMathOperator{\Diff}{Diff}
\DeclareMathOperator{\Fix}{Fix}
\DeclareMathOperator{\Homeo}{Homeo}
\DeclareMathOperator{\id}{id}
\DeclareMathOperator{\Int}{Int}
\DeclareMathOperator{\PSL}{PSL}
\DeclareMathOperator{\S^1}{S^1}

\DeclareMathOperator{\Stab}{Stab}

\newcommand{\C}{\ensuremath{\mathbb{C}}}
\newcommand{\R}{\ensuremath{\mathbb{R}}}

\newcommand{\Z}{\ensuremath{\mathbb{Z}}}

\begin{document}

\author{Yoshifumi Matsuda}

\address{Graduate School of Mathematical Sciences, 
University of Tokyo, 3-8-1 Komaba Meguro, 
Tokyo 153-8914, Japan}

\email{ymatsuda@ms.u-tokyo.ac.jp}



\title[Groups of real analytic diffeomorphisms of the circle]
{Groups of real analytic diffeomorphisms of the circle 
with a finite image under the rotation number function}

\thanks{The author is partially supported by Research Fellowships of 
the Japan Society for the Promotion of Science for Young Scientists.}

\begin{abstract} 
We consider groups of orientation-preserving 
real analytic diffeomorphisms of the circle 
which have a finite image under the rotation number function. 
We show that if such a group is nondiscrete 
with respect to the $C^1$-topology 
then it has a finite orbit. 
As a corollary, we show that if such a group has no finite orbit 
then each of its subgroups contains 
either a cyclic group of finite index 
or a nonabelian free subgroup. 
\end{abstract}

\maketitle

\section{Introduction}
The rotation number is a conjugacy invariant of orientation-preserving 
homeomorphisms of the circle. 
The precise definition of the rotation number will be given in Section 2. 
H. Poincar\'e showed that this invariant contains much information 
about dynamics. 
We quote only a part of his result which is interesting for us. 

\begin{thmnonum}[Poincar\'e]
Let $f$ be an orientation-preserving 
homeomorphism of the circle. 
Then the rotation number $\rho(f)$ of $f$ is rational 
if and only it has a finite orbit. 
More precisely, $\rho(f)=\frac{p}{q}$ \rm{mod} $\Z$ 
for certain coprime integers $p, q \in \Z$ 
if and only if $f$ has an orbit with $q$ points. 
\end{thmnonum}

The rotation number defines a map $\rho$ 
from the group $\Homeo_+(\S^1)$ of orientation-preserving homeomorphisms of 
the circle to $\R/\Z$, which we call the rotation number function. 
We are interested in the relation between 
the behavior of the rotation number function on subgroups of $\Homeo_+(\S^1)$ 
and their dynamics. 
The above theorem implies the following corollary. 

\begin{coro}
Let $\Gamma$ be a subgroup of $\Homeo_+(\S^1)$. 
If $\Gamma$ has a finite orbit, 
then it has a finite image under the rotation number function.
\end{coro}

It follows from the above theorem that 
the converse of Corollary 1.1 holds true 
for every cyclic subgroup of $\Homeo_+(\S^1)$. 
More generally, it holds true for every subgroup of $\Homeo_+(\S^1)$ 
preserving a probability measure on the circle (see Proposition 2.4). 

However the converse of Corollary 1.1 does not hold true in general 
even if we assume that $\Gamma$ is a subgroup of 
the group $\Diff^{\omega}_+(\S^1)$ of orientation-preserving 
real analytic diffeomorphisms of the circle. 
In fact, there exists a subgroup of $\Diff^{\omega}_+(\S^1)$ 
which has a finite image under the rotation number function 
and has no finite orbit. 

The projective group $\PSL(2, \R)$ acts naturally on the real projective line. 
By identifying the real projective line with the circle, 
we regard $\PSL(2, \R)$ as a subgroup of $\Diff^{\omega}_+(\S^1)$. 
It follows from Selberg's Lemma (see Lemma 8 of \cite{Sel}) that 
every finitely generated and discrete subgroup of $\PSL(2,\R)$ 
has a finite image under the rotation number function. 
Furthermore, there exist such subgroups 
which have no finite orbit in the circle such as $\PSL(2, \Z)$. 

On the other hand, after several works of 
J. Nielsen, W. Fenchel and A. Selberg, 
T. J{\o}rgensen gave a criterion for subgroups of $\PSL(2, \R)$ to 
be Fuchsian groups. 
In our context, his criterion is expressed as follows 
(In fact, he obtained a stronger result. See Theorem 2 of \cite{Jo}):

\begin{thmnonum}[J\o rgensen] 
Let $\Gamma$ be a subgroup of $\PSL(2, \R)$ 
which has no finite orbit in the circle. 
If $\Gamma$ has a finite image under the rotation number function, 
then it is a discrete subgroup of $\PSL(2, \R)$.
\end{thmnonum}

This theorem implies that the converse of Corollary 1.1 holds true 
for nondiscrete subgroups of $\PSL(2, \R)$. 
In this paper, we show that the converse of Corollary 1.1 holds true 
for subgroups of $\Diff^{\omega}_+(\S^1)$ which are nondiscrete 
with respect to the $C^1$-topology. 
Our main result is the following:

\begin{theo}
Let $\Gamma$ be a subgroup of $\Diff^{\omega}_+(\S^1)$ 
which is nondiscrete with respect to the $C^1$-topology. 
Then $\Gamma$ has a finite image under the rotation number function
if and only if it has a finite orbit.
\end{theo}

Before describing an application of Theorem 1.2, 
we recall the following result due to G. Margulis 
(see Theorem 3 of \cite{Ma}):

\begin{thmnonum}[Margulis]
Let $\Gamma$ be a subgroup of $\Homeo_+(\S^1)$. 
Then either 
there exists a $\Gamma$-invariant probability measure on the circle or
$\Gamma$ contains a nonabelian free subgroup. 
\end{thmnonum}

(For the proof, see \cite{Ma} pages 670-674. 
See also \cite{Gh} pages 360-363.) 
Applying Theorem 1.2, we can improve Margulis' theorem 
for certain subgroups of $\Diff^{\omega}_+(\S^1)$. 
The statement is as follows:

\begin{coro}
Let $\Gamma$ be a subgroup of $\Diff^{\omega}_+(\S^1)$
which has no finite orbit. 
Assume that $\Gamma$ has a finite image under the rotation number function. 
Then we have the following: 
\begin{enumerate}
    \item $\Gamma$ is discrete with respect to the $C^1$-topology. 
    \item For every subgroup $\Gamma'$ of $\Gamma$, 
          one of the following occurs: 
    \begin{enumerate}
        \item $\Gamma'$ is a finite group. 
        \item $\Gamma'$ contains an infinite cyclic subgroup of finite index. 
        \item $\Gamma'$ contains a nonabelian free subgroup.
    \end{enumerate}
\end{enumerate}
\end{coro}

It is known that every hyperbolic group has the same property 
as stated in the second assertion of Corollary 1.3 
 (see \cite{G-H} page 157). 
Hence it would be interesting to know 
if every finitely generated subgroup of $\Diff^{\omega}_+(\S^1)$ 
satisfying the assumption of Corollary 1.3 is a hyperbolic group. 
To the best of my knowledge, 
every known example of subgroups of $\Diff^{\omega}_+(\S^1)$ 
satisfying the assumption of Corollary 1.3 
is conjugate to an extension of a finite group by a Fuchsian group . 

We also show that neither Theorem 1.2 nor Corollary 1.3 holds true 
if we replace $\Diff^{\omega}_+(\S^1)$ by the group $\Diff^{\infty}_+(\S^1)$ 
of orientation-preserving $C^{\infty}$-diffeomorphisms of the circle 
(see Remark 4.4). 

There are two main ingredients to prove Theorem 1.2. 
The first ingredient is the fact that certain subgroups of $\Homeo_+(\S^1)$ 
contain an element whose fixed point set is nonempty and ``small''
(see Corollary 2.7). 
This fact results from an observation of E. Ghys, 
which is based on an argument of G. Margulis (see Proposition 2.6). 
In Section 2 we describe this fact together with several known facts 
about dynamics of subgroups of $\Homeo_+(\S^1)$. 
The second ingredient is the existence of certain local vector fields 
associated to nondiscrete subgroups of $\Diff^{\omega}_+(\S^1)$. 
(see Proposition 3.9). 
This fact is essentially due to I. Nakai \cite{Nak} and J. Rebelo \cite{Re}. 
The detailed argument is given in Section 3. 
In Section 4 we prove Theorem 1.2 and Corollary 1.3.

\section{Dynamics of groups of homeomorphisms of the circle}
We begin this section 
by recalling the definition of the rotation number. 
Denote by $\widetilde{\Homeo_+}(\S^1)$ 
the group of homeomorphisms of the real line $\R$ 
which commute with integer translations. 
Then we can lift every element of $\Homeo_+(\S^1)$ to 
$\widetilde{\Homeo_+}(\S^1)$ and 
two such lifts differ by an integer translation.
Let $\tilde{f}$ be an element of $\widetilde{\Homeo_+}(\S^1)$. 
Then we take a point $x$ of the real line and 
define the $\textit{translation number}$ $\tau(\tilde{f})$ of $\tilde{f}$ 
as follows: 
\begin{eqnarray*}
\tau(\tilde{f}) = \lim_{n \to \infty} \dfrac{\tilde{f}^n(x) - x}{n}. 
\end{eqnarray*}
We can prove that the limit on the right hand side always exists and 
does not depend on the choice of the point $x$. 
It follows from the definition that 
if two elements of $\widetilde{\Homeo_+}(\S^1)$ differ 
by an integer translation 
then their translation numbers differ by an integer. 

Now we consider an element $f$ of ${\Homeo_+}(\S^1)$. 
Then the translation numbers of its lifts in $\widetilde{\Homeo_+}(\S^1)$ 
differ by integers and hence the element 
\begin{eqnarray*}
\rho(f) = \tau(\tilde{f})\ \rm{mod} \Z \in \R/\Z
\end{eqnarray*}
is well-defined. This number $\rho(f)$ is called 
the $\textit{rotation number}$ of the homeomorphism $f$. 

To describe a property of the rotation number, 
we recall the following definition: 

\begin{defi}
\rm{
A map $h$ from the circle to itself is called 
an \textit{increasing continuous map of degree one} 
if it is covered by an increasing continuous map 
from the real line to itself 
which commutes with integer translations.

Let $f$ and $g$ be two elements of $\Homeo_+(\S^1)$. 
We say that $f$ is \textit{semi-conjugate} to $g$ 
if there exists an increasing continuous map $h$ of degree one 
from the circle to itself such that $hf = gh$.
}
\end{defi}

Note that $h$ need not to be a homeomorphism in the above definition. 
The following lemma shows 
that the rotation number is invariant under semi-conjugation 
(see Lemma 5.1.3 of \cite{H-H} for the proof):

\begin{lemm}
Let $f$ and $g$ be two elements of $\Homeo_+(\S^1)$. 
If $f$ is semi-conjugate to $g$, 
then we have $\rho(f)=\rho(g)$. 
\end{lemm}

The following proposition describes the possibilities 
for dynamics of subgroups of $\Homeo_+(\S^1)$. 
(see Proposition 5.6 of \cite{Gh} for the proof):

\begin{prop}
Let $\Gamma$ be a subgroup of $\Homeo_+(\S^1)$. 
Then exactly one of the following cases occurs:
\begin{enumerate}
\item $\Gamma$ has a finite orbit. 
\item All orbits are dense: 
      in this case $\Gamma$ is said to be \textit{minimal}. 
\item There exists a minimal set $C$ 
      which is homeomorphic to a Cantor set: 
      in this case this set $C$ is unique and is called 
      the \textit{exceptional minimal set}.
\end{enumerate}
\end{prop}

Now we show that 
the converse of Corollary 1.1 holds true for 
every subgroup of $\Homeo_+(\S^1)$ 
preserving a probability measure on the circle.

\begin{prop}
Let $\Gamma$ be a subgroup of $\Homeo_+(\S^1)$. 
Assume that 
there exists a $\Gamma$-invariant probability measure on the circle. 
Then $\Gamma$ has a finite image under the rotation number function 
if and only if it has a finite orbit.
\end{prop}

\begin{proof}
Let $\Gamma$ be a subgroup of $\Homeo_+(\S^1)$ 
which preserves a probability measure on the circle. 
Then it suffices to show that if $\Gamma$ has no finite orbit 
then it has an infinite image under the rotation number function. 

If $\Gamma$ is minimal, 
then every invariant probability measure has full support 
and has no atom. 
Hence every invariant probability measure is 
the image of the Lebesgue measure 
under a certain homeomorphism of the circle. 
Therefore $\Gamma$ is conjugate to a group 
which preserves the Lebesgue measure, 
that is, a group consisting of rotations. 
In particular, the rotation number function is injective on $\Gamma$. 
Since $\Gamma$ has no finite orbit, 
it is an infinite group and hence has an infinite image 
under the rotation number function. 

If $\Gamma$ has an exceptional minimal set, 
we can show that 
there exists a homomorphism $\psi$ from $\Gamma$ to $\Homeo_+(\S^1)$ 
such that $\gamma$ is semi-conjugate to $\psi(\gamma)$ 
for every element $\gamma$ in $\Gamma$ 
and $\psi(\Gamma)$ is minimal 
(see Proposition 5.8 of \cite{Gh}). 
Since $\psi(\Gamma)$ preserves a probability measure on the circle 
as well as $\Gamma$, 
the above argument implies that 
$\psi(\Gamma)$ has an infinite image under the rotation number function. 
Then it follows from Lemma 2.2 that 
$\Gamma$ also has an infinite image under the rotation number function. 

Thus we have finished the proof of the proposition.
\end{proof}

Now we consider minimal subgroups of $\Homeo_+(\S^1)$ 
which preserve no probability measure on the circle. 
Such subgroups have a certain type of dynamics. 
To see this, we begin by giving the following definition: 

\begin{defi}
\rm{
Let $\Gamma$ be a subgroup of $\Homeo_+(\S^1)$. 
A closed interval $I$ in the circle is said to be 
\textit{$\Gamma$-contractible} 
if there exists a sequence $\{g_n\}$ of elements of $\Gamma$ 
such that the length of the interval $g_n(I)$ tends to zero. 
}
\end{defi}

The following proposition is based on an argument of G. Margulis 
(see Section 4 of \cite{Ma}) and was formulated by E. Ghys 
(see \cite{Gh} pages 361-363). 

\begin{prop}
Let $\Gamma$ be a minimal subgroup of $\Homeo_+(\S^1)$. 
Assume that there exists no $\Gamma$-invariant probability measure 
on the circle. 
Then there exists an element $\theta$ of $\Homeo_+(\S^1)$ 
satisfying the following conditions: 
\begin{enumerate}
    \item $\theta$ is a periodic homeomorphism 
          which commutes with every element of $\Gamma$.
    \item For every point $x$ in the circle, 
          every closed interval contained in $[x, \theta(x)[$ is 
          $\Gamma$-contractible, 
          where $[x, \theta(x)[$ stands for the whole circle 
          if $\theta$ is the identity. 
\end{enumerate}
\end{prop}

\begin{proof}
We first claim that $\Gamma$ is not equicontinuous. 
Indeed, if it is equicontinuous, 
then it follows from Ascoli-Arzela's theorem that 
there exists a $\Gamma$-invariant probability measure on the circle, 
which contradicts the assumption. 

Now we claim that there exists a $\Gamma$-contractible interval in the circle. 
Indeed, since $\Gamma$ is not equicontinuous, 
there exist a sequence $\{I_n\}$ of open intervals in the circle, 
a sequence $\{g_n\}$ of elements of $\Gamma$ and 
a positive real number $\varepsilon > 0$ such that 
the length of $I_n$ tends to zero and 
the length of $g_n(I_n)$ is greater than $\varepsilon$. 
By taking a subsequence of $\{I_n\}$ if necessary, 
we may assume that there exists an open interval $I$ in the circle 
which is contained in $g_n(I_n)$ for every sufficiently large $n$. 
Then it follows that 
$g_n^{-1}(I)\subset I_n$ for every sufficiently large $n$. 
Therefore the length of $g_n^{-1}(I)$ tends to zero and 
hence the interval $I$ is $\Gamma$-contractible.

We denote by $\pi : \R \to \S^1$ the universal cover of the circle. 
For each point $x$ in the real line $\R$, we set 
\begin{eqnarray*}
\tilde{\theta}(x)=\sup\{y \in \R|
\textrm{
$\pi([x, y])$ is a $\Gamma$-contractible closed interval
}\}. 
\end{eqnarray*}
Thus we obtain a map $\tilde{\theta}$ from the real line to itself. 
Since $\tilde{\theta}$ commutes with integral translations 
and hence it induces a map $\theta$ from the circle to itself. 
Note that $\theta$ commutes with every element of $\Gamma$. 

Now we claim that the map $\theta$ is an element of $\Homeo_+(\S^1)$. 
Indeed, the map $\tilde{\theta}$ is increasing. 
If it were not strictly increasing, 
then the union of open intervals in the circle 
on which the map $\theta$ is constant would be 
a nonempty, open and $\Gamma$-invariant subset of the circle. 
The minimality of $\Gamma$ implies that this set is the whole circle. 
Hence the image of $\theta$ is a single point and 
every element of $\Gamma$ fixes this point. 
However this contradicts the minimality of $\Gamma$ and 
hence the map $\tilde{\theta}$ is strictly increasing. 
Similarly, if the map $\theta$ were not continuous, 
then the interior of the complement of the image $\theta(\S^1)$ 
would be the whole circle. 
Therefore the map $\theta$ is continuous 
and we conclude that it is an element of $\Homeo_+(\S^1)$.

Next we claim that the homeomorphism $\theta$ is periodic. 
Indeed, if $\theta$ had a unique exceptional minimal set, 
then this set would also be an exceptional minimal set for $\Gamma$ 
and this contradicts the minimality of $\Gamma$. 
If $\theta$ is minimal, 
it would be conjugate to an irrational rotation. 
Since every element of $\Gamma$ commutes with $\theta$ and
a homeomorphism which commutes with an irrational rotation is a rotation, 
$\Gamma$ is conjugate to a group of rotations. 
However this contradicts the existence of a $\Gamma$-contractible interval. 
Therefore $\theta$ has a finite orbit. 
The union of periodic points of $\theta$ is 
a closed, $\Gamma$-invariant and nonempty subset of the circle. 
Then it follows from the minimality of $\Gamma$ that 
the homeomorphism $\theta$ is periodic. 

Moreover it follows from the definition of $\theta$ that 
the second condition in the proposition is satisfied 
and thus we have finished the proof of the proposition. 
\end{proof}

Proposition 2.6 yields the following corollary, 
which plays an important role in the proof of Theorem 1.2. 

\begin{coro}
Let $\Gamma$ and $\theta$ be as in Proposition 2.6.
Let $\kappa$ denote the period of $\theta$ and 
let $I_0$ be a closed interval in the circle such that 
the intervals $I_0$, $\theta(I_0)$, \ldots, $\theta^{\kappa-1}(I_0)$
are mutually disjoint. 
Then there exists an element $h$ of $\Gamma$ such that 
\begin{eqnarray*}
\emptyset \ne \Fix(h) 
& \subset & \textstyle\bigcup_{j=0}^{\kappa-1}\Int(\theta^j(I_0)) 
\ \textrm{and} \\
h(\S^1 \setminus \textstyle\bigcup_{j=0}^{\kappa-1}\Int(\theta^j(I_0))) & 
\subset & \textstyle\bigcup_{j=0}^{\kappa-1}\Int(\theta^j(I_0)). 
\end{eqnarray*}
\end{coro}

\begin{proof}
We first claim that there exists an element $h_0$ of $\Gamma$ 
which is different from the identity and has a fixed point. 
Indeed, if there were no such element, it follows from H\"older's theorem 
(see Theorem 6.10 of \cite{Gh}) that $\Gamma$ would be abelian and 
would preserve a probability measure on the circle. 

Let $x$ be a point in the complement of $\Fix(h_0)$. 
Then there exists a $\Gamma$-contractible closed interval $J$ in 
$[x, \theta(x)[$ such that $\Fix(h_0) \cap [x, \theta(x)[ \subset J$. 
It follows that $\Fix(h_0) \subset \bigcup_{j=0}^{\kappa-1} \theta^j(J)$. 
Since $J$ is $\Gamma$-contractible and $\Gamma$ is minimal, 
there exists an element $g$ of $\Gamma$ such that $g(J) \subset \Int(I_0)$. 
Now we put $h=gh_0g^{-1}$. 
Then we have 
\begin{eqnarray*}
\Fix(h)=g(\Fix(h_0)) 
\subset \textstyle\bigcup_{j=0}^{\kappa-1} \theta^j(g(J)) 
\subset \textstyle\bigcup_{j=0}^{\kappa-1} \Int(\theta^j(I_0)). 
\end{eqnarray*}
Since $h$ is a translation on each connected component of $\Fix(h)$, 
replacing $h$ by its power if necessary, 
we may assume that 
\begin{eqnarray*}
h(\S^1 \setminus \textstyle\bigcup_{j=0}^{\kappa-1}\Int(\theta^j(I_0))) 
\subset \textstyle\bigcup_{j=0}^{\kappa-1}\Int(\theta^j(I_0)). 
\end{eqnarray*}
Thus we have finished the proof of the corollary.
\end{proof}

\section{Locally nondiscrete subgroups and associated local vector fields}
We begin this section by giving the definition of the local nondiscreteness 
for subgroups of $\Diff^{\omega}_+(\S^1)$. 
\begin{defi}
\rm{
For $0 \le r \le \infty$, 
a subgroup $\Gamma$ of $\Diff^{\omega}_+(\S^1)$ is said to be 
\textit{locally nondiscrete with respect to the $C^r$-topology} 
if there exists a sequence of elements of $\Gamma \setminus \{\id\}$ 
which converges to the identity on an open interval in the circle 
with respect to the $C^r$-topology.
}
\end{defi}

Locally nondiscrete subgroups of $\Diff^{\omega}_+(\S^1)$ are imposed 
restrictions on their dynamics as the following proposition shows. 

\begin{prop}
Let $\Gamma$ be a subgroup of $\Diff^{\omega}_+(\S^1)$ 
which is locally nondiscrete with respect to the $C^0$-topology. 
Then $\Gamma$ has no exceptional minimal set.
\end{prop}

\begin{proof}
We prove by contradiction. 
Suppose that $\Gamma$ has an exceptional minimal set $C$ 
and is locally nondiscrete with respect to the $C^0$-topology. 
Then there exist an open interval $I$ in the circle 
and a sequence $\{g_n\}$ of elements of $\Gamma \setminus \{\id\}$ 
which converges to the identity on $I$ 
with respect to the $C^0$-topology. 
Let $J$ be a connected component of 
the complement of $C$ which intersects $I$ 
and let $a$ be an endpoint of $J$.

We claim that the stabilizer $\Stab_a(\Gamma)$ of $a$ is nontrivial. 
Indeed, the assumption implies that for every $n$ sufficiently large, 
the image $g_n(J)$ intersects the component $J$. 
On the other hand, the image $g_n(J)$ is 
a connected component of the complement of $C$ as well as $J$. 
Therefore, for every $n$ sufficiently large, 
$g_n$ preserves the connected component $J$ 
and fixes its endpoints. 

Since $\Gamma$ has an exceptional minimal set, 
it follows from a theorem of G. Hector 
that the stabilizer $\Stab_a(\Gamma)$ of $a$ must be cyclic 
(see page 461 of \cite{Nav}). 
Hence there exists no sequence of elements of 
$\Stab_a(\Gamma) \setminus \{\id\}$ which converges to the identity 
on $I \cap J$. 
This contradicts the fact that for every $n$ sufficiently large, 
$g_n$ belongs to $\Stab_a(\Gamma) \setminus \{\id\}$. 
Thus we have finished the proof of the proposition.
\end{proof}

In the sequel, we associate certain local vector fields to 
locally nondiscrete subgroups of $\Diff^{\omega}_+(\S^1)$. 
The local vector fields we associate have the property 
described in the following definition.

\begin{defi}
\rm{
Let $\Gamma$ be a subgroup of $\Diff^{\omega}_+(\S^1)$ and 
let $J$ be an open interval in the circle. 
For $0 \le r \le \infty$, 
a local vector field $X$ defined on $J$ is said to be 
\textit{in the $C^r$-closure of $\Gamma$ relative to $I$} 
if the following condition is satisfied: 
for every compact subinterval $J_0$ of $J$ and 
for every positive real number $t_0 > 0$ such that the local flow $\varphi^t$ 
associated to $X$ is defined on $J_0$ for every $0 \leq t \leq t_0$, 
the map $\varphi^{t_0}$ is the $C^r$-limit of 
the restriction to $J_0$ of a sequence of elements of $\Gamma$.
}
\end{defi}

The following lemma is used to show the existence of a local vector field 
which is in the $C^0$-closure for certain nondiscrete subgroups of 
$\Diff^{\omega}_+(\S^1)$ (cf. Proposition 3.7 and 3.9). 

\begin{lemm}
Let $\{g_n\}$ be a sequence of $C^1$-maps 
from an open interval $I$ in the real line $\R$ into $\R$ 
and let $\{\lambda_n\}$ be a sequence  
of positive real numbers which diverges to the infinity. 
Assume that there exist positive real numbers $A_1$, $A_2$ and $A_3$ 
satisfying the following: 
\begin{eqnarray}
A_1 
\le \inf_{x \in I} \lambda_n|(g_n-\id)(x)| 
\le \sup_{x \in I} \lambda_n|(g_n-\id)(x)| 
& \le & A_2 \quad \textrm{and} \\
\sup_{x \in I} \lambda_n|(g_n-\id)'(x)| 
& \le & A_3. 
\end{eqnarray}
Then for each open and relatively compact subinterval $J$ of $I$ 
there exists a nowhere vanishing $C^0$-vector field $X$ on $J$ 
which satisfies the following condition:
for every compact subinterval $J_0$ of $J$ and 
for every positive real number $t_0 > 0$ such that the local flow $\varphi^t$ 
associated to $X$ is defined on $J_0$ for every $0 \le t \le t_0$, 
the sequence $\{g_n^{[\lambda_nt_0]}\}$ 
(where $[\lambda_nt_0]$ stands for the integral part of $\lambda_nt_0$) 
converges to the map $\varphi^{t_0}$ uniformly on $J_0$.  
\end{lemm}

\begin{proof}
Our proof is based on the argument 
in the proof of Proposition 3.5 of \cite{Nak}. 
For each positive integer $n \ge 1$, 
we define a vector field $X_n$ on $I$ by 
\begin{eqnarray*}
X_n(x)=\lambda_n(g_n-\id)(x)\frac{\partial}{\partial x}.
\end{eqnarray*}
Then it follows from the assumption that 
the family $\{\lambda_n(g_n-\id)\ |\ n \ge 1\}$ 
is uniformly bounded and equicontinuous 
as a family of real valued functions on $I$. 
Let $J$ be an open and relatively compact subinterval of $I$. 
Then by applying Ascoli-Arzela's theorem and 
taking a subsequence if necessary, 
we may assume that the sequence of vector fields $\{X_n\}$ converges to 
a $C^0$-vector field $X$ on $J$ with respect to the $C^0$-topology. 
Moreover the inequality (1) implies that 
$\inf_{x \in J}||X(x)|| \ge A_1 > 0$ 
and hence $X$ is a nowhere zero vector field on $J$. 

Now we show that the vector field $X$ satisfies the desired condition. 
Let $J_0$ be a compact subinterval of $J$ and 
let $t_0 > 0$ be a positive real number such that 
the local flow $\varphi^t$ associated to $X$ is defined on $J_0$ 
for $0 \le t \le t_0$. 
Let $I'$ be an open and relatively compact subinterval of $I$ 
containing $J$. 
Then there exists a positive real number $\delta > 0$ such that 
the $\delta$-neighborhood of each point in $I'$ is contained in $I$. 
By taking a subsequence if necessary, 
we may assume that 
\begin{eqnarray}
\frac{1}{2}A_2\lambda_n^{-1} \{(1+A_3\lambda_n^{-1})^j-1\} < \delta 
\end{eqnarray}
for every positive integer $n$ 
and every integer $0 \le j \le [\lambda_n t_0]$.  
We denote by $\varphi_n^t$ the local flow associated to $X_n$. 
By taking a subsequence if necessary, 
we may assume that 
the point $\varphi_n^t(x)$ belongs to the interval $I'$ 
for every positive integer $n$, every real number $0 \le t \le t_0$ 
and every point $x$ in $J_0$. 
Now we prove the following claim: \\

\noindent
\textbf{Claim}\ 
For every positive integer $n$, every integer $0 \le j \le [\lambda_n t]$ 
and every point $x$ in $J_0$, we have 
\begin{eqnarray*}
|\varphi_n^{\lambda_n^{-1}j}(x)-g_n^j(x)| 
\le \frac{1}{2} A_2\lambda_n^{-1}\{(1+A_3\lambda_n^{-1})^j-1\} 
\end{eqnarray*}
and the point $g_n^j(x)$ belongs to the interval $I$. 

\begin{proof}
The proof is done by the induction on $j$. 
The claim obviously holds true for $j=0$. 
Now let us fix a positive integer $n$ arbitrarily
and assume that the result holds 
for an integer $j$ with $0 \le j \le [\lambda_n t]-1$. 
Let $x$ be an arbitrary point in $J_0$. 
Then it follows from the inequality (1) that 
\begin{eqnarray*}
& & \int_{0}^{\lambda_n^{-1}} |\varphi_n^{t+\lambda_n^{-1}j}(x)-g_n^j(x)| dt \\
& \le & \int_{0}^{\lambda_n^{-1}} 
        \{ |(\varphi_n^t-\id)(\varphi_n^{\lambda_n^{-1}j}(x))| 
        + |\varphi_n^{\lambda_n^{-1}j}(x)-g_n^j(x)| \} dt \\ 
& \le & \int_{0}^{\lambda_n^{-1}} \biggl\{ 
        \int_{0}^{t} \lambda_n|(g_n-\id)(\varphi_n^{s+\lambda_n^{-1}j}(x))| ds          + |\varphi_n^{\lambda_n^{-1}j}(x)-g_n^j(x)| \biggl\} dt \\
& \le & \int_{0}^{\lambda_n^{-1}} (A_2t 
        +|\varphi_n^{\lambda_n^{-1}j}(x)-g_n^j(x)|) dt \\ 
& \le & \frac{1}{2} A_2 \lambda_n^{-2}
        + \lambda_n^{-1} |\varphi_n^{\lambda_n^{-1}j}(x)-g_n^j(x)|. 
\end{eqnarray*}
Since the two points $\varphi_n^{\lambda_n^{-1}j}(x)$ and $g_n^j(x)$ 
belong to the interval $I$ by the assumption, 
the inequality (2) and the above inequality imply that
\begin{eqnarray*}
& & |((\varphi_n^{\lambda_n^{-1}}-\id)(\varphi_n^{\lambda_n^{-1}j}(x))
    - (g_n-\id)(g_n^j(x))| \\
& \le & \int_{0}^{\lambda_n^{-1}} 
        \lambda_n |(g_n-\id)(\varphi_n^{t+\lambda_n^{-1}j}(x)) 
        - (g_n-\id)(g_n^j(x))| dt \\
& \le & (\sup_{y \in I} \lambda_n|(g_n-\id)'(y)|)
        \int_{0}^{\lambda_n^{-1}} 
        |\varphi_n^{t+\lambda_n^{-1}j}(x)-g_n^j(x)| dt \\
& \le & \frac{1}{2}A_2A_3\lambda_n^{-2}
      + A_3\lambda_n^{-1} |\varphi_n^{\lambda_n^{-1}j}(x)-g_n^j(x)|. 
\end{eqnarray*}
Hence it follows from the assumption that
\begin{eqnarray*}
& & |\varphi_n^{\lambda_n^{-1}(j+1)}(x)-g_n^{j+1}(x)| \\
& \le & |((\varphi_n^{\lambda_n^{-1}}-\id)(\varphi_n^{\lambda_n^{-1}j}(x))
    - (g_n-\id)(g_n^j(x))| + |\varphi_n^{\lambda_n^{-1}j}(x)-g_n^j(x)| \\
& \le & \frac{1}{2} A_2A_3\lambda_n^{-2}
      + (1+A_3\lambda_n^{-1}) |\varphi_n^{\lambda_n^{-1}j}(x)-g_n^j(x)| \\
& \le & \frac{1}{2} A_2\lambda_n^{-1}\{(1+A_3\lambda_n^{-1})^{j+1}-1\}. 
\end{eqnarray*}
Thus we have proved the first assertion of the claim for $j+1$. 

Moreover, since the point $\varphi_n^{\lambda_n^{-1}(j+1)}(x)$ 
belongs to the interval $I'$ by the assumption, 
the inequality (3) and the property of $\delta$ imply that 
the point $g_n^{j+1}(x)$ belongs to the interval $I$. 
Thus we have proved the second assertion of the claim for $j+1$ 
and have finished the proof of the claim. 
\end{proof}

Now we return to the proof of the lemma. 
Since the sequence $\{\varphi_n^{t_0}\}$ uniformly converges to 
$\varphi^{t_0}$ on $J_0$ 
and the sequence $\{\lambda_n^{-1}[\lambda_nt_0]\}$ converges to $t_0$, 
the inequality (1) implies that 
the sequence $\{\varphi_n^{\lambda_n^{-1}[\lambda_nt_0]}\}$ 
uniformly converges to $\varphi^{t_0}$ on $J_0$. 
On the other hand, it follows from the above claim that 
\begin{eqnarray*}
|\varphi_n^{\lambda_n^{-1}[\lambda_nt_0]}(x)-g_n^{[\lambda_nt_0]}(x)|
& \le & \frac{1}{2} A_2\lambda_n^{-1} 
        \{(1+A_3\lambda_n^{-1})^{[\lambda_nt_0]} -1\} \\
& \le & \frac{1}{2} A_2\lambda_n^{-1}(e^{A_3t_0}-1) 
\end{eqnarray*}
for every positive integer $n$ and every point $x$ in $J_0$. 
This implies that the sequence $\{g_n^{[\lambda_nt_0]}\}$ 
also uniformly converges to $\varphi^{t_0}$ on $J_0$. 
Thus we have finished the proof of the lemma.
\end{proof}

Now we quote a result of I.Nakai (see Section 3 of \cite{Nak}). 
He showed that we can associate local vector fields to 
pseudogroups of holomorphic diffeomorphisms 
on open neighborhoods of the origin $0$ 
in the complex plane $\C$ fixing the origin. 
Precisely, his result is stated as follows:

\begin{prop}[I. Nakai]
Let $f$ and $g$ be holomorphic diffeomorphisms 
on a neighborhood of $0$ in $\C$. 
Assume that $f$ and $g$ have the following Taylor expansions: 
\begin{eqnarray*}
f(z)=z+az^i+\cdots , \ g(z)=z+bz^j+\cdots , \quad 
a, b \ne 0 , \ 1 \le i < j.
\end{eqnarray*}
For every positive integer $n \ge 1$, 
we put $\lambda_n = n^{\frac{j-i}{i}}$ and $g_n = f^{-n}gf^n$. 
We denote by $B_f$ the basin of $f$, that is, 
the set of the points $z$ in the domain of $f$ 
such that $f^n(z)$ converges to $0 \in \C$ as $n$ tends to the infinity. 
Then we have the following: 
\begin{enumerate}
    \item The sequence $\{\lambda_n(g_n - \id)\}$ 
          converges (locally uniformly) to a holomorphic vector field $X$ 
          on $B_f \setminus \{0\}$. 
    \item For every relatively compact subset $V$ of $B_f \setminus \{0\}$ 
          and for every positive real number $t_0 > 0$ such that 
          the local flow $\varphi^t$ associated to $X$ is defined on $V$ 
          for every $0 \leq t \leq t_0$, 
          the sequence $\{g_n^{[\lambda_nt_0]}\}$ 
          converges to the map $\varphi^{t_0}$ uniformly on $V$.  
\end{enumerate}
\end{prop}

When we embed the real line $\R$ in the complex plane $\C$, 
every real analytic diffeomorphism defined on a neighborhood of $0$ in $\R$ 
has a holomorphic extension to a certain neighborhood of $0$ in $\C$. 
Therefore we obtain the following corollary:

\begin{coro}
Let $f$ and $g$ be real analytic diffeomorphisms 
on a neighborhood of $0$ in $\R$. 
Assume that $f$ and $g$ have the following Taylor expansions: 
\begin{eqnarray*}
f(x)=x+ax^i+\cdots , \ g(x)=x+bx^j+\cdots , \quad 
a, b \ne 0 , \ 1 \le i < j.
\end{eqnarray*}
For every positive integer $n \ge 1$, 
we put $\lambda_n = n^{\frac{j-i}{i}}$ and $g_n = f^{-n}gf^n$. 
We denote by $B_f$ the basin of $f$, that is, 
the set of the points $x$ in the domain of $f$ 
such that $f^n(x)$ converges to $0 \in \R$ as $n$ tends to the infinity. 
Then we have the following: 
\begin{enumerate}
    \item The sequence $\{\lambda_n(g_n - \id)\}$ 
          converges (locally uniformly) to a real analytic vector field $X$ 
          on $B_f \setminus \{0\}$. 
    \item For every relatively compact subset $I_0$ of $B_f \setminus \{0\}$ 
          and for every positive real number $t_0 > 0$ such that 
          the local flow $\varphi^t$ associated to $X$ is defined on $I_0$ 
          for every $0 \leq t \leq t_0$, 
          the sequence $\{g_n^{[\lambda_nt_0]}\}$ 
          converges to the map $\varphi^{t_0}$ uniformly on $I_0$.  
\end{enumerate}
\end{coro}

Using this corollary, we show the following proposition:

\begin{prop}
Let $\Gamma$ be a subgroup of $\Diff^{\omega}_+(\S^1)$ 
and let $p$ be a point in the circle. 
Assume that the stabilizer $\Stab_p(\Gamma)$ of the point $p$ in $\Gamma$ 
is not cyclic. 
Then we have the following:
\begin{enumerate}
    \item $\Gamma$ is locally nondiscrete with respect to the $C^1$-topology 
          and 
    \item there exists a local $C^0$-vector field defined on an open interval 
          which is in the $C^0$-closure of $\Gamma$. 
\end{enumerate}
\end{prop}

\begin{proof}
The proof is divided into two cases 
in accordance with the stabilizer $\Stab_p(\Gamma)$  
being abelian. \\

\noindent
Case\ 1. The stabilizer $\Stab_p(\Gamma)$ of $p$ in $\Gamma$ is abelian. 

Let $f$ be an element of $\Stab_p(\Gamma)$ 
which is different from the identity. 
We claim that 
every element $g$ of $\Stab_p(\Gamma)$ fixes every point of $\Fix(f)$. 
Indeed, $g$ commutes with $f$ 
and hence it preserves $\Fix(f)$. 
Moreover, it follows from the real analyticity of $f$ that 
$\Fix(f)$ is a finite set. 
Since $g$ fixes the point $p$ in $\Fix(f)$, 
this implies that $g$ fixes every point of $\Fix(f)$. 

Let $I=]a, b[$ be a connected component of the complement of $\Fix(f)$. 
Then it follows from a result of G. Seekers that 
there exists a $C^1$-vector field $X$ on $[a, b[$ 
with the following properties: 
\begin{itemize}
     \item $X$ is nowhere zero on $I$,
     \item the local flow $\varphi^t$ associated to $X$ verifies 
           $\varphi^1 = f|_{[a,b[}$ and 
     \item the centralizer of $f|_{[a,b[}$ in $\Diff^1_+([a,b[)$ is 
           equal to $\{\varphi^t\}$
\end{itemize}
(see \cite{Sz}. See also Theorem 2.2 of \cite{Ser}). 
Since $\Stab_p(\Gamma)$ is abelian and noncyclic, 
the restrictions of its elements to $I$ form a dense subgroup of 
$\{\varphi^t\}$. 
This implies that there exists a sequence of elements of 
$\Stab_p(\Gamma) \setminus \{\id\}$ 
which converges to the identity on $I$ with respect to the $C^1$-topology. 
Moreover it follows that the restriction of the vector field $X$ to $I$ 
is in the $C^1$-closure of $\Gamma$ with respect to $I$. 
Thus we have finished the proof in this case. \\

\noindent
Case\ 2. The stabilizer $\Stab_p(\Gamma)$ of $p$ in $\Gamma$ is not abelian.

In this case, there exist elements $f$ and $g$ of $\Stab_p(\Gamma)$ 
which have the following Taylor expansions in a local coordinate $x$ 
with $p=0$: 
\begin{eqnarray*}
f(x)=x+ax^{i+1}+\cdots , \ g(x)=x+bx^{j+1}+\cdots , \ 
a, b \ne 0 , \ 0 \le i < j.
\end{eqnarray*}

If $i \ge 1$, then the claim of the proposition results from Corollary 3.6. 
Hence it suffices to consider the case where $i=0$. 
We put $\lambda=f'(0)(=1+a)$. 
Replacing $f$ by its inverse if necessary, 
we may assume that $0< \lambda < 1$. 
Then it follows from a theorem of G. Koenigs (see Theorem 2.1 of \cite{C-G}) 
that there exists a local real analytic coordinate $x$ 
on an open neighborhood $I_1$ of $p$ 
such that $p=0$ and $f$ is written in the form $f(x)=\lambda x$ on $I_1$. 
Via the local coordinate $x$, we regard each element of $\Stab_p(\Gamma)$ 
as a local orientation-preserving real analytic diffeomorphism on $I_1$ 
fixing the point 0. 

Replacing $g$ by its inverse if necessary, we may assume that $b < 0$. 
Then there exists a positive real number $\delta_1 >0$ such that 
the closed interval $[0, \delta_1]$ is contained in $I_1$ and 
$g^{(j+1)}(y) < 0$ for every $y$ in $[0, \delta_1]$. 
We denote by $I_1'$ the closed interval $[0, \delta_1]$. 
We put $M_1=\sup_{y \in I_1'} |g^{(j+1)}(y)|$ and 
$M_2=\inf_{y \in I_1'} |g^{(j+1)}(y)|$. 
Then we have $0 < M_2 \le M_1$. 
We take a positive real number $\delta_2$ in $]0, \delta_1[$ 
and denote by $I$ the open interval $]\delta_2, \delta_1[$. 

For each positive integer $n \ge 1$, 
we define an element $g_n$ of $\Stab_p(\Gamma)$ by $g_n=f^{-n}gf^n$. 
Now we prove the following lemma: 

\begin{lemm}
For every positive integer $n \ge 1$ and every point $x$ in $I$, we have 
\begin{eqnarray}
\frac{\delta_2^{j+1}}{(j+1)!}M_2 
\le \lambda^{-jn} |(g_n-\id)(x)|
& \le & \frac{\delta_1^{j+1}}{(j+1)!}M_1 
\ \textrm{and} \\
\lambda^{-jn}|(g_n-\id)'(x)| 
& \le & \frac{\delta_1^j}{j!}M_1.
\end{eqnarray}
\end{lemm}

\begin{proof}
We fix a point $x$ in $I$ arbitrarily. 
First we claim that 
\begin{eqnarray}
g_n(x)=\lambda^{-n}g(\lambda^n x). 
\end{eqnarray}
Indeed, since $g^{(j+1)}(y) < 0$ for every point $y$ in $I$, 
it follows from the definition of $j$ and Taylor's theorem that 
$0 < g(f^n(x)) < f^n(x) < \lambda^n \delta_1$. 
Hence the point $g(f^n(x))$ belongs to the open interval 
$]0, \lambda^n\delta_1[$. 
Since $f^{-n}(z)=\lambda^{-n}z$ 
for every point $z$ in $]0, \lambda^n\delta_1[$, 
we obtain the equality (6). 

Next we prove the inequality (4). 
The equality (6) implies that 
\begin{eqnarray*}
|(g_n-\id)(x)|=\lambda^{-n}|(g-\id)(\lambda^n x)|.
\end{eqnarray*}
Moreover, it follows from Taylor's theorem that 
there exists a point $y$ in $I_1'$ such that 
\begin{eqnarray*}
|(g-\id)(\lambda^n x)|
=\frac{(\lambda^nx)^{j+1}}{(j+1)!} |g^{(j+1)}(\lambda^n y)|.
\end{eqnarray*}
These two equalities imply the inequality (4). 

Finally we prove the inequality (5). 
It follows from the equality (6) that 
\begin{eqnarray*}
|(g_n-\id)'(x)|=|(g-\id)'(\lambda^n x)|.
\end{eqnarray*}
Since $g'(0)=1$ by the assumption, it follows from Taylor's theorem that 
there exists a point $y'$ in $I_1'$ such that 
\begin{eqnarray*}
|(g-\id)'(\lambda^n x)|
=\frac{(\lambda^n x)^j}{j!} |g^{(j+1)}(\lambda^n y')|.
\end{eqnarray*}
This equality implies the inequality (5) 
and we have finished the proof of the lemma.
\end{proof}

We return to the proof of Proposition 3.7 in Case 2. 
It follows from Lemma 3.8 that 
the sequence $\{g_n\}$ converges to the identity on $I$.
Moreover the sequence $\{g_n\}$ satisfies the assumption of Lemma 3.4 
(we put $\lambda_n=\lambda^{-jn}$). 
Therefore Lemma 3.4 implies the existence of a desired local vector field. 
Thus we have finished the proof in Case 2 and 
the proof of the proposition has been completed.
\end{proof}

The rest of this section is devoted to proving the following proposition: 

\begin{prop}
Let $\Gamma$ be a subgroup of $\Diff^{\omega}_+(\S^1)$ 
which is nondiscrete with respect to the $C^1$-topology. 
Assume that $\Gamma$ preserves no probability measure 
on the circle. 
Then there exist a nowhere zero $C^0$-vector field 
which is defined on an open interval in the circle 
and is in the $C^0$-closure of $\Gamma$.
\end{prop}

\begin{proof}
Since $\Gamma$ has no finite orbit by the assumption, 
it is minimal by Proposition 3.2. 
Moreover $\Gamma$ contains a nontrivial element which has a fixed point. 
Indeed if there were no such element, it follows from H\"older's theorem 
that $\Gamma$ would be abelian and 
would preserve a probability measure on the circle. 
Then by a result of E. Ghys there exists an element $f$ of $\Gamma$ 
and a point $p$ in the circle such that $f(p)=p$ and $f'(p)=\lambda<1$ 
(see Theorem 1.2.7 of \cite{E-T}). 
If the stabilizer $\Stab_p(\Gamma)$ is not cyclic, 
then Corollary 3.6 implies the claim of the proposition. 
Hence we may assume that the stabilizer $\Stab_p(\Gamma)$ is cyclic. 

Since $\Gamma$ is locally nondiscrete 
with respect to the $C^1$-topology
and minimal, 
there exists a sequence $\{h_n\}$ of elements of $\Gamma$ 
which converges to the identity on a neighborhood $I_1$ of $p$ 
with respect to the $C^1$-topology. 
Since the stabilizer $\Stab_p(\Gamma)$ is cyclic by the assumption, 
by taking a subsequence if necessary, 
we may assume that $h_n(p) \ne p$ for every positive integer $n \ge 1$. 
Moreover, by taking a subinterval of $I_1$ if necessary, 
we may assume that 
there exists a local real analytic coordinate $x$ defined on $I_1$ 
such that $p=0$ and $f$ is written in the form $f(x)=\lambda x$. 
Furthermore, we may assume that in our coordinate $x$ 
the neighborhood $I_1$ of 0 contains the closed interval $[-1, 1]$. 

Now we show that there exist an open subinterval $I$ of $I_1$ 
and a sequence $\{g_n\}$ of elements of $\Gamma$ 
satisfying the assumption of Lemma 3.4. 
We follow an argument due to J. Rebelo (see Section 3 of \cite{Re}). 
For an open subset $J \subset \R$ and 
a $C^1$-map defined on $J$ onto its image in $\R$, we denote 
$||h||_{1,J} = \sup_{x \in J}(|h(x)|+|h'(x)|)$. 

\begin{lemm}
Let $I_2$ be an open and relatively compact subinterval of $I_1$. 
Assume that we are given a positive integer $n \ge 1$ and 
a positive real number $\varepsilon > 0$. 
Then there exists a positive real number $\delta > 0$ 
such that  
\begin{eqnarray*}
|(h^n-\id)(x)-n(h-\id)(x)| \le \varepsilon |(h-\id)(x)|
\end{eqnarray*}
for every map $h$ satisfying $||h-\id||_{1,I_1} < \delta$ 
and every point $x$ in $I_2$.
\end{lemm}

\begin{proof}
We prove the lemma by induction on $n$. 
The lemma obviously holds true for $n=1$ and 
we assume that the result holds true for $n \ge 1$. 
Then there exists a positive real number $\delta' > 0$ such that 
\begin{eqnarray}
|(h^n-\id)(x)-n(h-\id)(x)| \le \dfrac{\varepsilon}{2} |(h-\id)(x)|
\end{eqnarray}
for every map $h$ satisfying $||h-\id||_{1,I_1} < \delta'$ 
and every point $x$ in $I_2$. 
Let $d > 0$ be a positive real number such that 
the open interval $]x-d, x+d[$ is contained in $I_1$ 
for every point $x$ in $I_2$. 
Now we put 
$\delta=\min\{\frac{d}{n}, \frac{\varepsilon}{2n+\varepsilon}, \delta'\}$. 
We take a $C^1$-map $h : I \to \R$ satisfying $||h-\id||_{1,I_1} < \delta$ 
and a point $x$ in $I_2$ arbitrarily. 
Then for every $1 \le m \le n$ we have 
\begin{eqnarray*}
|(h^m-\id)(x)| \le \sum_{l=0}^{m-1}|(h^{l+1}-h^l)(x)| < \frac{md}{n} \le d.
\end{eqnarray*}
In particular, the point $h^n(x)$ belongs to $I_1$. 
Therefore it follows from the mean value theorem and the inequality (7) that 
\begin{eqnarray*}
& & |(h^{n+1}-\id)(x)-(n+1)(h-\id)(x)| \\
& \le & |(h^{n+1}-h^n)(x))-(h-\id)(x)|+|(h^n-\id)(x)-n(h-\id)(x))| \\
& \le & (\sup_{y \in I_1}|(h-\id)'(y)|)|(h^n-\id)(x)|+
\frac{\varepsilon}{2}|(h-\id)(x)| \\
& \le & \frac{\varepsilon}{2n+\varepsilon}(n+\frac{\varepsilon}{2})|(h-\id)(x)|
+ \frac{\varepsilon}{2}|(h-\id)(x)| = \varepsilon |(h-\id)(x)|.
\end{eqnarray*}
Thus the result holds true for $n+1$ 
and we have finished the proof of the lemma. 
\end{proof} 

\begin{lemm}
Let $I_2$ be an open and relatively compact subinterval of $I_1$. 
Assume that we are given a positive integer $n \ge 1$ and 
positive real numbers $C_1$ and $C_2$. 
Then there exists a positive real number $\delta > 0$ such that 
\begin{eqnarray*}
|(h-\id)(y)-(h-\id)(x)| < C_1|(h-\id)(x)|
\end{eqnarray*}
for every $C^1$-map $h:I_1 \to \R$ satisfying $||h-\id||_{1,I_1} < \delta$ 
and two arbitrary points $x, y$ in $I_2$ 
satisfying $|y-x| < C_2|(h^n-\id)(x)|$.
\end{lemm}

\begin{proof}
Let $\delta' > 0$ be a positive real number such that 
Lemma 3.10 holds true for $I_2$, $n$ and $\varepsilon = 1$. 
We set $\delta=\min\{\delta', \frac{C_1}{C_2(n+1)}\}$.
We take a $C^1$-map $h:I_1 \to \R$ satisfying $||h-\id||_{1,I_1} < \delta$ 
and two points $x, y$ in $I_2$ 
satisfying $|y-x| < C_2|(h^n-\id)(x)|$ arbitrarily. 
Then it follows from the mean value theorem and Lemma 3.10 that
\begin{eqnarray*}
& &|(h-\id)(y)-(h-\id)(x)| \\
& \le & (\sup_{z \in I_1} |(h-\id)'(z)|)|y-x| 
\le \delta C_2|(h^n-\id)(x)| \\
& \le & \delta C_2(n+1)|(h-\id)(x)| 
\le C_1|(h-\id)(x)|.
\end{eqnarray*}
Thus we have finished the proof of the lemma. 
\end{proof}

\begin{lemm}
Let $n \ge 1$ be a positive integer. 
Then there exists a positive real number $\delta(n) > 0$ with
\begin{eqnarray}
(n+1)\delta(n) & < & 1-\lambda
\end{eqnarray}
such that for every $C^1$-map $h:I_1 \to \R$ 
satisfying $||h-\id||_{1,I_1} < \delta(n)$ and $h(0) > 0$ 
there exists a positive integer $k \ge 1$ 
satisfying the following: 
\begin{eqnarray}
(n+1)|(h-\id)(\lambda^k)| 
& > & \lambda^k(1-\lambda), \\ 
n|(h-\id)(\lambda^k)| 
& < & \lambda^{k-1}(1-\lambda) \quad and \\
|(h-\id)(y)-(h-\id)(\lambda^k)| 
& < & \frac{1}{4}|(h-\id)(\lambda^k)| 
\end{eqnarray}
for every point $y$ in $I_1$ with $|y-\lambda^k| < \lambda^{k-1}(1-\lambda)$.
\end{lemm}

\begin{proof}
Replacing in the statement of Lemma 3.11 $I_2$ by $]-1,1[$, $n$ by $n+2$, 
$C_1$ by $\frac{1}{4(n+1)}$ and $C_2$ by $\lambda^{-1}$, 
we see that there exists a positive real number $\delta(n)>0$ 
such that if a $C^1$-map $h:I_1 \to \R$ satisfies 
$||h-\id||_{1,I_1}<\delta(n)$ then we have
\begin{eqnarray}
|(h-\id)(y)-(h-\id)(x)| 
& < & \frac{1}{4(n+1)}|(h-\id)(x)|
\end{eqnarray}
for two arbitrary points $x, y$ in $]-1,1[$ 
with $|y-x| < \lambda^{-1}|(h^{n+2}-\id)(x)|$. 
Moreover, replacing in the statement of Lemma 3.10 
$n$ by $n+2$ and $\varepsilon$ by $1$ 
and reducing $\delta(n)$ if necessary, 
we may assume that the inequality (8) holds true 
for every positive integer $n \ge 1$ and
\begin{eqnarray}
|(h^{n+2}-\id)(x)-(n+2)(h-\id)(x)| < |(h-\id)(x)|
\end{eqnarray}
for every point $x$ in $]-1,1[$. 

We fix a $C^1$-map $h:I_1 \to \R$ 
satisfying $||h-\id||_{1,I_1}<\delta(n)$ and 
$h(0) > 0$ arbitrarily. 
Then we can consider the smallest positive integer $k \ge 1$ verifying 
the inequality (9). 
Then the inequalities (9) and (13) imply that 
\begin{eqnarray}
\qquad\quad
|\lambda^{k-1}-\lambda^k| 
< \lambda^{-1}(n+1)|(h-\id)(\lambda^k)| 
< \lambda^{-1}|(h^{n+2}-\id)(\lambda^k)|. \!\!\!\!\!
\end{eqnarray}
Therefore, setting $x=\lambda^k$ and $y=\lambda^{k-1}$ 
in the inequality (12), we obtain
\begin{eqnarray}
|(h-\id)(\lambda^{k-1})-(h-\id)(\lambda^k)| 
& < & \frac{1}{4(n+1)}|(h-\id)(\lambda^k)|
\end{eqnarray}

Now we claim that the inequality (10) holds true. 
Indeed, if it were false it would follow from the inequality (15) that
\begin{eqnarray*}
(n+1)|(h-\id)(\lambda^{k-1})| 
& > & (n+\frac{3}{4})|(h-\id)(\lambda^k)| \\
& > & n|(h-\id)(\lambda^k)| 
\ge \lambda^{k-1}(1-\lambda)
\end{eqnarray*}
and this would contradict the minimality of $k$. 

Finally we claim that the inequality (11) holds true. 
Indeed, replacing $x$ by $\lambda^k$ in the inequality (12), 
it suffices to see that 
$\lambda^{k-1}(1-\lambda) < \lambda^{-1}|(h^{n+2}-\id)(\lambda^k)|$, 
which is obtained from the inequality (14).
Thus we have finished the proof of the lemma.
\end{proof}

Recall that we have a sequence $\{h_n\}$ of $C^1$-maps 
defined on $I_1$ which converges to the identity on $I_1$ 
with respect to the $C^1$-topology. 
Since $h_n(0) \ne 0$ for every  positive integer $n \ge 1$ by assumption, 
replacing $h_n$ by its inverse if necessary, 
we may assume that $h_n(0) > 0$. 
By taking a subsequence if necessary, we may also assume that 
$||h_n-\id||_{1,I_1} < \delta(n)$ for every positive integer $n \ge 1$. 
Then it follows from Lemma 3.12 that 
for each positive integer $n \ge 1$ 
there exists a positive integer $k_n \ge 1$ 
satisfying the following:
\begin{eqnarray}
(n+1)|(h_n-\id)(\lambda^{k_n})| 
& > & \lambda^{k_n}(1-\lambda), \\ 
n|(h_n-\id)(\lambda^{k_n})| 
& < & \lambda^{k_n-1}(1-\lambda) \quad \textrm{and} \\
|(h_n-\id)(y)-(h_n-\id)(\lambda^{k_n})| 
& < & \frac{1}{4}|(h_n-\id)(\lambda^{k_n})| 
\end{eqnarray}
for every point $y$ in $I_1$ 
with $|y-\lambda^{k_n}|<\lambda^{k_n-1}(1-\lambda)$. 

We set $\lambda'=\max\{0, 2\lambda-1\}$ and 
define a compact subinterval $I$ of $I_1$ by $I=[\lambda', \lambda]$. 
For each positive integer $n \ge 1$, 
we define an element $g_n$ of $\Gamma$ by 
$g_n=f^{-k_n+1}h_nf^{k_n-1}$.

\begin{lemm}
For every positive integer $n \ge 1$ and every point $x$ in $I$, we have 
\begin{eqnarray}
\frac{3}{8}\lambda(1-\lambda) 
\le n|(g_n-\id)(x)|
& \le & \frac{5}{4}(1-\lambda) \quad \textrm{and} \\
n|(g_n-\id)'(x)| 
& < & 1-\lambda.
\end{eqnarray}
\end{lemm}

\begin{proof}
First we claim that 
\begin{eqnarray}
g_n(x)=\lambda^{-k_n+1}h_n(\lambda^{k_n-1}x).
\end{eqnarray} 
Indeed, we have $0 < x < \lambda$ 
and hence the inequality (17) implies that
\begin{eqnarray*}
h_n(f^{k_n-1}(x)) 
< h_n(\lambda^{k_n}) 
< \lambda^{k_n} + \frac{1}{n}\lambda^{k_n-1}(1-\lambda) 
\le \lambda^{k_n-1}. 
\end{eqnarray*}
Moreover we have 
\begin{eqnarray*}
h_n(f^{k_n-1}(x)) > h_n(0) >0. 
\end{eqnarray*}
Hence the point $h_n(f^{k_n-1}(x))$ belongs to the open interval 
$]0, \lambda^{k_n-1}[$. 
Since $f^{-k_n+1}(y)=\lambda^{-k_n+1}y$ 
for every point $y$ in $]0, \lambda^{k_n-1}[$, 
we obtain the equality (21).

Next we prove the inequality (19). 
It follows from the equality (21) that 
\begin{eqnarray}
n|(g_n-\id)(x)|=\lambda^{-k_n+1}n|(h_n-\id)(\lambda^{k_n-1}x)|.
\end{eqnarray}
Note that $|x-\lambda| < 1-\lambda$ 
and hence $|\lambda^{k_n-1}x-\lambda^{k_n}| < \lambda^{k_n-1}(1-\lambda)$. 
Therefore, the inequality (18) implies that 
\begin{eqnarray}
\qquad\quad
\frac{3}{4}n|(h_n-\id)(\lambda^{k_n})|
<n|(h_n-\id)(\lambda^{k_n-1}x)|
<\frac{5}{4}n|(h_n-\id)(\lambda^{k_n})|. \!\!\!\!\!\!\!
\end{eqnarray}
Moreover, it follows from the inequality (16) that 
\begin{eqnarray}
\qquad
n|(h_n-\id)(\lambda^{k_n})| 
\ge \frac{1}{2}(n+1)|(h_n-\id)(\lambda^{k_n})| 
> \frac{1}{2}\lambda^{k_n}(1-\lambda). \!\!\!
\end{eqnarray}
Combining the inequalities (17), (23), (24), we have 
\begin{eqnarray*}
\frac{3}{8}\lambda^{k_n}(1-\lambda)
< n|(h_n-\id)(\lambda^{k_n-1}x)| 
< \frac{5}{4}\lambda^{k_n-1}(1-\lambda). 
\end{eqnarray*}
In view of the equality (22), 
this implies the inequality (19).

Finally we claim that the inequality (20) holds true. 
Indeed, it follows from the equality (21) and the inequality (8) that 
\begin{eqnarray*}
n|(g_n-\id)'(x)| 
\le \sup_{y \in I} n|(h_n-\id)'(\lambda^{k_n-1}y)| 
\le n\delta(n) < 1-\lambda. 
\end{eqnarray*}
Thus we have finished the proof of the lemma.
\end{proof}

In view of Lemma 3.13, Lemma 3.4 implies 
the existence of a desired vector field on $I$. 
Thus we have finished the proof of Proposition 3.9. 
\end{proof}

\section{Proof of Theorem 1.2 and Corollary 1.3}
We first prove Theorem 1.2. 
In view of Corollary 1.1 and Proposition 2.4, 
it suffices to show the following proposition: 

\begin{prop} 
Let $\Gamma$ be a subgroup of $\Diff^{\omega}_+(\S^1)$ 
which is locally nondiscrete with respect to the $C^1$-topology. 
Assume that there exists no $\Gamma$-invariant probability measure 
on the circle. 
Then $\Gamma$ has an infinite image under the rotation number function. 
\end{prop}

Note that in Proposition 4.1 
we only assume that $\Gamma$ is \textit{locally} nondiscrete 
with respect to the $C^1$-topology. 
The assumption implies that $\Gamma$ has no finite orbit 
and hence it is minimal by Proposition 3.2. 
Then we can take a homeomorphism $\theta$ in $\Homeo_+(\S^1)$ 
which satisfies the conditions in Proposition 2.6. 
We denote by $\kappa$ the period of $\theta$. 

Furthermore by Proposition 3.9, 
there exist an open interval $I$ in the circle and 
a nowhere zero local vector field $X$ on $I$ 
which is in the $C^0$-closure of $\Gamma$ relative to $I$. 
We take a closed subinterval $I_0$ of $I$ such that 
there exists a positive real number $t_0 > 0$ 
such that the local flow $\varphi^t$ associated to $X$ is defined on $I_0$ 
for $-t_0 \leq t \leq t_0$ and 
$\varphi^{t_0}(I_0)$ does not intersect $I_0$. 
Replacing $I_0$ by its closed subinterval if necessary, 
the intervals $I_0$, $\theta(I_0)$, \ldots, $\theta^{\kappa-1}(I_0)$
are mutually disjoint. 
Then it follows from Corollary 2.7 that 
there exists an element $h$ of $\Gamma$ such that 
\begin{eqnarray*}
\emptyset \ne \Fix(h) 
& \subset & \textstyle\bigcup_{j=0}^{\kappa-1}\Int(\theta^j(I_0)) 
\ \textrm{and} \\
h(\S^1 \setminus \textstyle\bigcup_{j=0}^{\kappa-1}(\theta^j(I_0)) & 
\subset & \textstyle\bigcup_{j=0}^{\kappa-1}\Int(\theta^j(I_0)). 
\end{eqnarray*}

Let $\tilde{h}$ denote the lift of $h$ under the covering projection 
$\pi : \R \to \S^1$ which has fixed points and
let $\tilde{\theta}$ denote the lift of $\theta$ 
under the covering projection $\pi : \R \to \S^1$ 
such that $\tilde{\theta}^{\kappa}(x)=x+1$ for every point $x$ in $\R$. 
Note that the homeomorphism $\tilde{\theta}$ commutes with 
every lift of any element of $\Gamma$. 
Let $\tilde{I}$ be a connected component of $\pi^{-1}(I)$ and 
let $\tilde{I}_0$ be the connected component of $\pi^{-1}(I_0)$ 
which is contained in $\tilde{I}$. 
We put $\tilde{I}_0=[a, b]$. 
The properties of $h$ imply that 
\begin{eqnarray}
\Fix(\tilde{h}) 
& \subset & \textstyle\bigcup_{j \in \Z}
\Int(\tilde{\theta}^j(\tilde{I}_0)) 
\ \textrm{and} \\ 
\tilde{h} (\R \setminus \textstyle\bigcup_{j \in \Z} 
(\tilde{\theta}^j(\tilde{I}_0)) 
& \subset & \textstyle\bigcup_{j \in \Z}
\Int(\tilde{\theta}^j(\tilde{I}_0)). 
\end{eqnarray}
Note that the relations (25) and (26) still hold 
if we replace $\tilde{h}$ by its inverse. 
It follows from (25) that
\begin{eqnarray*}
\Fix(\tilde{h}) \cap [a, \tilde{\theta}(a)[ 
\subset ]a, b[ = \Int(\tilde{I}_0). 
\end{eqnarray*}
By replacing $h$ by its inverse if necessary, 
we may assume that $\tilde{h}(a) < a$. 
Then it follows that 
$\tilde{h}(\tilde{\theta}(a)) < \tilde{\theta}(a)$.
Since $\tilde{h}$ has a fixed point in $]a, b[$ 
and has no fixed point in $[b, \tilde{\theta}(a)[$, 
we have
\begin{eqnarray}
a < \tilde{h}(b) < b. 
\end{eqnarray}
Since the closed interval $[b, \tilde{\theta}(a)]$ is a connected component 
of the complement of 
$\textstyle\bigcup_{j \in \Z} \Int(\tilde{\theta}^j(\tilde{I}_0))$, 
it follows from (26) and (27) that
\begin{eqnarray*}
\tilde{h}([b, \tilde{\theta}(a)]) \subset ]a,b[ = \Int(\tilde{I}_0). 
\end{eqnarray*}
In particular we have 
\begin{eqnarray}
\tilde{h}(\tilde{\theta}(a)) < b.
\end{eqnarray}

We denote by $\tilde{X}$ the pullback of the local vector field $X$ 
to $\tilde{I}$.
Then the local flow $\tilde{\varphi}^t$ associated to $\tilde{X}$ 
is defined on $\tilde{I}_0$ for $-t_0 \leq t \leq t_0$. 
Since it follows from the assumption that 
$\tilde{\varphi}^{t_0}(\tilde{I}_0)$ and $\tilde{I}_0$ are disjoint, 
replacing $X$ by $-X$ if necessary, 
we may assume that $\tilde{\varphi}^{t_0}(a) > b$. 
Since $X$ is nowhere zero, 
this implies that 
\begin{eqnarray*}
\tilde{\varphi}^t(x) > x 
\end{eqnarray*}
for every positive real number $t \le t_0$ 
and every point $x$ in $\tilde{I}_0$. 
Now we consider the real number $T \in \R$ defined as follows:
\begin{eqnarray*}
T=\sup\{t \geq 0\ |\ \tilde{\varphi}^t(x)=\tilde{h}(x)\ 
\textrm{for a certain point $x$ in $\tilde{I}_0$}\}. 
\end{eqnarray*}
Since $\tilde{h}(b) < b < \tilde{\varphi}^{t_0}(a)$, 
this number $T$ is well-defined and we have $T < t_0$. 
It follows from the definition of $T$ that 
\begin{eqnarray}
\tilde{h}(x) \leq \tilde{\varphi}^T(x) < \tilde{h}(x)+1 \ 
\textrm{for every point $x$ in $\tilde{I}_0$}. 
\end{eqnarray}
Moreover, there exists a point $x_0$ in $\tilde{I}_0$ 
such that $\tilde{\varphi}^T(x_0) = \tilde{h}(x_0)$. 
Since we have $\tilde{h}(a) < a$ and $\tilde{h}(b) < b$, 
this point $x_0$ belongs to the interior of $\tilde{I}_0$. 
Therefore we have
\begin{eqnarray}
\tilde{\varphi}^T(a) < \tilde{\varphi}^T(x_0) 
= \tilde{h}(x_0) < \tilde{h}(b). 
\end{eqnarray}

We take a decreasing sequence $\{t_i\}_{i=1}$ of positive real numbers 
with $t_1 < t_0$ which converges to $T$. 
Then for every positive integer $i \ge 1$ there exists a sequence 
$\{\tilde{g}_{i,j}\}_{j=1}^{\infty}$ of lifts of elements of $\Gamma$ 
under the covering projection $\pi : \R \to \S^1$ 
which converges to $\tilde{\varphi}^{t_i}$ uniformly on $\tilde{I}_0$ 
as $j$ tends to the infinity. 
By taking subsequences if necessary, we may assume that
\begin{eqnarray*}
\tilde{g}_{i,j}(x) > \tilde{\varphi}^T(x)
\end{eqnarray*}
for arbitrary two positive integers $i \ge 1$, $j \ge 1$, 
and every point $x$ in $\tilde{I}_0$. 
Now we put $\tilde{g}_i=\tilde{g}_{i,i}$ for every positive integer $i \ge 1$. 
Then it follows that 
the sequence $\{\tilde{g}_i\}$ converges to $\tilde{\varphi}^T$ 
uniformly on $\tilde{I}_0$ 
and for every positive integer $i \ge 1$ we have
\begin{eqnarray}
\tilde{g}_i(x) > \tilde{\varphi}^T(x) > x \ 
\textrm{for every point $x$ in $\tilde{I}_0$}. 
\end{eqnarray}
In view of the inequalities (29) and (30), 
by taking a subsequence if necessary, 
we may assume that for every positive integer $i \ge 1$ we have 
\begin{eqnarray}
\tilde{h}(x) < \tilde{g}_i(x) < \tilde{h}(x)+1 \ 
\textrm{for every point $x$ in $\tilde{I}_0$ and} \\
\tilde{g}_i(a) < \tilde{h}(b). 
\end{eqnarray}

We put $\tilde{f}_i=\tilde{h}^{-1}\tilde{g}_i$ 
for every positive integer $i \ge 1$. 
Now we estimate the translation number 
$\tau(\tilde{f}_i)$ of $\tilde{f}_i$. 

\begin{lemm}
\begin{enumerate}
    \item We have $0 < \tau(\tilde{f}_i) <1$
          for every positive integer $i \ge 1$. 
    \item The sequence $\{\tau(\tilde{f}_i)\}$ converges to zero 
          as $i$ tends to the infinity.
\end{enumerate}
\end{lemm}

\begin{proof}
(i) We take a point $x$ in $]b, \tilde{\theta}(a)[$ arbitrarily. 
Then it follows from the inequalities (28) and (31) that 
\begin{eqnarray*}
\tilde{h}(x) < \tilde{h}(\tilde{\theta}(a)) < b < \tilde{g}_i(b). 
\end{eqnarray*}
Furthermore the inequality (33) implies that 
\begin{eqnarray*}
\tilde{g}_i(\tilde{\theta}(a)) = \tilde{\theta}(\tilde{g}_i(a)) 
< \tilde{\theta}(\tilde{h}(b)) \le \tilde{h}(b)+1 < \tilde{h}(x)+1. 
\end{eqnarray*}
Thus we have 
\begin{eqnarray*}
\tilde{h}(x) < \tilde{g}_i(b) 
< \tilde{g}_i(x) 
< \tilde{g}_i(\tilde{\theta}(a)) < \tilde{h}(x)+1. 
\end{eqnarray*}
Together with the inequality (32), this implies that 
\begin{eqnarray*}
\tilde{h}(x) < \tilde{g}_i(x) < \tilde{h}(x)+1
\end{eqnarray*}
for every point $x$ in $[a,\tilde{\theta}(a)[$. 

For every point $x$ in $\R$, there exists an integer $j \in \Z$ such that 
the point $\tilde{\theta}^j(x)$ belongs to the interval 
$[a,\tilde{\theta}(a)[$. 
Since every lift of any element of $\Gamma$ commutes with $\tilde{\theta}$, 
we have 
\begin{eqnarray*}
\tilde{h}(x) < \tilde{g}_i(x) < \tilde{h}(x) +1. 
\end{eqnarray*}
This implies that $x < \tilde{f}_i(x) < x+1$. 
Since $x$ is an arbitrary point in $\R$, we have $0 < \tau(\tilde{f}_i) < 1$. 
Thus we have finished the proof of (i). \\

\noindent
(ii) Since the map $\tilde{h}^{-1}$ is continuous and 
commutes with integer translations, it is uniformly continuous. 
Hence for every positive real number $\varepsilon > 0$ 
there exists a positive real number $\delta_{h^{-1}}(\varepsilon) > 0$ 
such that for $|x-y| < \delta_{h^{-1}}(\varepsilon)$, we have 
\begin{eqnarray}
|\tilde{h}^{-1}(x)-\tilde{h}^{-1}(y)| < \varepsilon. 
\end{eqnarray}

We fix a positive real number $0<d<1$ such that 
$]x_0-d, x_0+d[ \subset \tilde{I}_0$. 
Now we show the following claim: \\

\noindent
\textbf{Claim}\ 
For every positive integer $n \ge 1$ and 
every positive real number $\varepsilon > 0$ with 
$\varepsilon < \frac{d}{n}$, 
there exists a positive real number $\delta(n,\varepsilon) > 0$ 
such that $\sup_{x \in \tilde{I}_0} 
|\tilde{g}_i(x)-\tilde{\varphi}^T(x)| < \delta(n,\varepsilon)$ 
implies
\begin{eqnarray*}
|\tilde{f}_i^j(x_0)-\tilde{f}_i^{j-1}(x_0)| < \varepsilon
\end{eqnarray*}
for $j=1,\ldots, n$.

\begin{proof}
The proof is done by induction on $n$. 
We first consider the case $n=1$. 
We fix a positive real number $\varepsilon > 0$ with $\varepsilon < d$ 
arbitrarily.
Now we put $\delta(1,\varepsilon) = \delta_{h^{-1}}(\varepsilon)$. 
If $\sup_{x \in \tilde{I}_0} |\tilde{g}_i(x)-\tilde{\varphi}^T(x)| 
< \delta(1,\varepsilon)$ then we obtain 
\begin{eqnarray*}
|\tilde{g}_i(x_0)-\tilde{h}(x_0)|
=|\tilde{g}_i(x_0)-\tilde{\varphi}^T(x_0)| < \delta(1,\varepsilon). 
\end{eqnarray*}
Hence it follows from the inequality (34) that 
\begin{eqnarray*}
|\tilde{f}_i(x_0)-x_0| 
= |\tilde{h}^{-1}(\tilde{g}_i(x_0)) - \tilde{h}^{-1}(\tilde{h}(x_0))| 
< \varepsilon
\end{eqnarray*}
and thus we have finished the proof for the case $n=1$. 

Next we assume that the claim is true for a positive integer $n \ge 1$. 
We take a positive real number $\varepsilon > 0$ 
with $\varepsilon < \frac{d}{n+1}$ arbitrarily. 
Then there exists a positive real number $\delta_1>0$ such that 
$\sup_{x \in \tilde{I}_0} |\tilde{g}_i(x)-\tilde{\varphi}^T(x)| < \delta_1$
implies 
\begin{eqnarray}
|\tilde{f}_i^j(x_0)-\tilde{f}_i^{j-1}(x_0)| < \varepsilon 
\end{eqnarray}
for $j=1, \ldots, n$. 
Since the map $\tilde{\varphi}^T$ is uniformly continuous 
on $\tilde{I}_0$, 
there exists a positive real number $\delta_2 >0$ such that 
for arbitrary two points $x, y$ in $\tilde{I}_0$ 
with $|x-y| < \delta_2$ 
we have
\begin{eqnarray}
|\tilde{\varphi}^T(x)-\tilde{\varphi}^T(y)| < 
\frac{1}{3}\delta_{h^{-1}}(\varepsilon).
\end{eqnarray}
Moreover by the assumption of the induction 
there exists a positive real number $\delta_3 >0$ such that 
$\sup_{x \in \tilde{I}_0} |\tilde{g}_i(x)-\tilde{\varphi}^T(x)| 
< \delta_3$ 
implies
\begin{eqnarray}
|\tilde{f}_i^n(x_0)-\tilde{f}_i^{n-1}(x_0)| < \delta_2.
\end{eqnarray}

Now we put 
$\delta(n,\varepsilon) 
= \min\{\frac{1}{3}\delta_{h^{-1}}(\varepsilon), \delta_1, \delta_3\}$. 
Then by the inequality (37), if 
$\sup_{x \in \tilde{I}_0} 
|\tilde{g}_i(x)-\tilde{\varphi}^T(x)|<\delta(n, \varepsilon)$ 
we have
\begin{eqnarray*}
|\tilde{f}_i^n(x_0)-\tilde{f}_i^{n-1}(x_0)| < \delta_2.
\end{eqnarray*}
Then it follows from the inequality (36) that 
\begin{eqnarray*}
& & |\tilde{g}_i(\tilde{f}_i^n(x_0)) - \tilde{g}_i(\tilde{f}_i^{n-1}(x_0))| \\
& \leq & |\tilde{\varphi}^T(\tilde{f}_i^n(x_0)) 
          - \tilde{\varphi}^T(\tilde{f}_i^{n-1}(x_0))| 
          + 2\sup_{x \in \tilde{I}_0} 
          |\tilde{g}_i(x)-\tilde{\varphi}^T(x)| \\
& < & \frac{1}{3}\delta_{h^{-1}}(\varepsilon)
      + \frac{2}{3}\delta_{h^{-1}}(\varepsilon) 
= \delta_{h^{-1}}(\varepsilon). 
\end{eqnarray*}
Therefore the inequality (34) implies that
\begin{eqnarray}
\qquad\quad
|\tilde{f}_i^{n+1}(x_0)-\tilde{f}_i^n(x_0)| 
 = |\tilde{h}^{-1}\tilde{g}_i(\tilde{f}_i^n(x_0))
- \tilde{h}^{-1}\tilde{g}_i(\tilde{f}_i^{n-1}(x_0))| < \varepsilon. 
\end{eqnarray}
Then it follows from the inequalities (35) and (38) that 
the claim is true for $n+1$. 
Thus we have finished the proof of the claim. 
\end{proof}

We complete the proof of (ii) of Lemma 4.2. 
For every positive integer $n \ge 1$ and 
every positive real number $\varepsilon > 0$ with $\varepsilon < \frac{d}{n}$, 
we have 
$\sup_{x \in \tilde{I}_0} |\tilde{g}_i(x)-\tilde{\varphi}^T(x)| 
< \delta(n,\varepsilon)$ 
for every sufficiently large positive integer $i \ge 1$. 
Then  it follows from the above claim that
for every such positive integer $i \ge 1$, we have
\begin{eqnarray*}
|\tilde{f}_i^n(x_0)-x_0| 
\le \sum_{j=1}^n |\tilde{f}_i^j(x_0)-\tilde{f}_i^{j-1}(x_0)| 
< n\varepsilon < d < 1.
\end{eqnarray*}
This implies that $|\tau(\tilde{f}_i)| \leq \frac{1}{n}$ 
and we have finished the proof of (ii). 
Thus we have finished the proof of the lemma.
\end{proof}

Now we complete the proof of Proposition 4.1. 
For each positive integer $i \ge 1$, 
the diffeomorphism $\tilde{f}_i$ on $\R$ determines 
an element $f_i$ of $\Gamma$. 
Then it follows from Lemma 4.2 that the rotation number $\rho(f_i)$ 
of $f_i$ is not equal to zero for every positive integer $i \ge 1$ and 
tends to zero as $i$ tends to the infinity. 
Thus we have finished the proof of Proposition 4.1 
and hence the proof of Theorem 1.2 has been completed. 

\begin{rema}
\rm{
In view of Selberg's lemma, we can show that 
every finitely generated subgroup of $\PSL(2, \R)$ 
satisfying the assumption of Proposition 4.1 
contains an element whose rotation number is irrational. 
Hence it would be interesting to know 
if this assertion still holds true 
when we replace $\PSL(2, \R)$ by $\Diff^{\omega}_+(\S^1)$.
}
\end{rema}

Next we prove Corollary 1.3. 
The first assertion immediately follows from Theorem 1.2. 
To prove the second assertion, 
we take a subgroup $\Gamma'$ of $\Gamma$ arbitrarily. 
Assume that $\Gamma'$ contains no subgroup 
isomorphic to a nonabelian free group. 
Then it follows from Margulis' theorem that 
there exists a $\Gamma'$-invariant probability measure on the circle. 
Since $\Gamma'$ has a finite image under the rotation number function 
as well as $\Gamma$, 
Proposition 2.4 implies that $\Gamma'$ has a finite orbit.
Hence there exists a point $p$ in the circle such that 
the stabilizer $\Stab_p(\Gamma')$ is a finite index subgroup of $\Gamma'$. 

Now we claim that $\Stab_p(\Gamma')$ is cyclic. 
Indeed, if it were not cyclic, then Proposition 3.7 would imply that 
$\Gamma'$ would be locally nondiscrete with respect to the $C^1$-topology 
and so would $\Gamma$. 
Hence it would follow from Proposition 4.1 that 
$\Gamma$ would have an infinite image under the rotation number, 
which contradicts the assumption. 
Thus we have finished the proof of Corollary 1.3.

\begin{rema}
\rm{
Neither Theorem 1.2 nor Corollary 1.3 holds true 
if we replace $\Diff^{\omega}_+(\S^1)$ by $\Diff^{\infty}_+(\S^1)$. 
To see this, we construct a finitely generated subgroup $\Gamma$ of 
$\Diff^{\infty}_+(\S^1)$ satisfying the following conditions: 
\begin{itemize}
\item $\Gamma$ is nondiscrete with respect to the $C^{\infty}$-topology, 
\item $\Gamma$ contains a subgroup isomorphic to $\Z \oplus \Z$, 
\item $\Gamma$ has a trivial image under the rotation number function, 
      that is, $\rho(\Gamma)=\{0\}$ and 
\item $\Gamma$ has no finite orbit. 
\end{itemize}

Let $\Gamma_1$ be a finitely generated and torsionfree Fuchsian group 
which has an exceptional minimal set $C$. 
We take a connected component $I$ of the complement of the minimal set $C$ 
and a nontrivial $C^{\infty}$-vector field $X$ on the circle which vanishes 
outside the interval $I$. 
Let $\Gamma_2$ be a dense subgroup of the flow 
associated to $X$ which is isomorphic to $\Z \oplus \Z$. 
Now we define $\Gamma$ to be the subgroup of $\Diff^{\infty}_+(\S^1)$
generated by $\Gamma_1$ and $\Gamma_2$. 

Then $\Gamma$ is nondiscrete 
with respect to the $C^{\infty}$-topology as well as $\Gamma_2$. 
To see that $\Gamma$ has a trivial image under the rotation number function, 
we claim that every element of $\Gamma$ has a fixed point in $C$. 
Indeed, for every element $\gamma$ of $\Gamma$, 
there exist elements $\gamma_{1,1}, \ldots, \gamma_{1,k}$ of $\Gamma_1$ 
and $\gamma_{2,1}, \ldots, \gamma_{2,k}$ of $\Gamma_2$ such that 
$\gamma=\gamma_{1,1} \gamma_{2,1} \cdots \gamma_{1,k} \gamma_{2,k}$. 
Then the product $\gamma_{1,1}\cdots\gamma_{1,k}$ is an element of $\Gamma_1$ 
and hence has a fixed point $x$ in $C$. 
Since the set $C$ is invariant under the action of $\Gamma_1$ 
and fixed pointwise by the action of $\Gamma_2$, 
it follows that $\gamma(x) = x$ and we have finished the proof of the claim. 
We can also see that the set $C$ is an exceptional minimal set for $\Gamma$ 
and hence $\Gamma$ has no finite orbit. 
Thus we proved that the group $\Gamma$ satisfies the desired conditions.
}
\end{rema}

\section*{Acknowledgements}
This paper is the main part of my doctoral thesis. 
I would like to thank my doctoral thesis adviser, Professor Takashi Tsuboi, 
for many helpful comments and proofreading of earlier versions of this paper. 
He suggested that we might improve an earlier version of Proposition 4.1, 
which only asserted the existence of an element with nonzero rotation number. 

I also would like to thank Professor Shigenori Matsumoto for helping me with 
some classical theory of dynamical systems on the circle. 
He suggested that an earlier version of Corollary 1.3 might hold true, 
which only asserted the nonexistence of a subgroup isomorphic to 
a free abelian group of rank two.
This suggestion motivated me to start this research. 

I also would like to thank Professor Julio Rebelo for discussing 
when I visited Pontificia Universidade Cat\'olica do Rio de Janeiro. 
He suggested that we might improve an earlier version of Theorem 1.2, 
which assumed the existence of a finite set of generators 
sufficiently close to the identity. 

\nocite{*}
\bibliographystyle{plain}
\bibliography{rotnum_arxiv-ref}

\begin{thebibliography}{10}

\bibitem{C-G}
L.~Carleson and T.~Gamelin.
\newblock {\em Complex Dynamics}.
\newblock Universitext: Tracts in Mathematics. Springer Verlag, New York, 1993.

\bibitem{E-T}
Y.~Eliashberg and W.~Thurston.
\newblock {\em Confoliations}, volume~13 of {\em University Lecture Series}.
\newblock Amer. Math. Soc., Providence, RI, 1998.

\bibitem{Gh}
E.~Ghys.
\newblock Groups acting on the circle.
\newblock {\em L'Enseignement Math\'ematique}, 47:329--407, 2001.

\bibitem{G-H}
E.~Ghys and P.~de~la Harpe.
\newblock {\em Sur les groups hyperboliques d'apr\`es Mikhael Gromov},
  volume~83 of {\em Progress in Mathematics}.
\newblock Birkh\"auser, Boston, 1990.

\bibitem{H-H}
G.~Hector and U.~Hirsch.
\newblock {\em Introduction to the Geometry of Foliations, Part A}.
\newblock Aspects of Mathematics. Friedr. Vieweg and Sohn, Braunschweig, 1981.

\bibitem{Jo}
T.~J{\o}rgensen.
\newblock A note on subgroups of $\mathrm{SL}(2, \mathbb{C})$.
\newblock {\em Quart. J. Math. Oxford Ser. II}, 28:209--212, 1977.

\bibitem{Ma}
G.~Margulis.
\newblock Free subgroups of the homeomorphism group of the circle.
\newblock {\em C. R. Acad. Sci. Paris S\'er. I Math.}, 9:669--674, 2000.

\bibitem{Nak}
I.~Nakai.
\newblock Separatrices for non solvable dynamics on $(\mathbb{C}, 0)$.
\newblock {\em Ann. Inst. Fourier}, 44:569--599, 1994.

\bibitem{Nav}
A.~Navas.
\newblock Sur les groupes de diff\'eomorphismes du cercle engendr\'es par des
  \'el\'ements proches des rotations.
\newblock {\em L'Enseignment Math\'ematique}, 50:29--68, 2004.

\bibitem{Re}
J.~Rebelo.
\newblock Ergodicity and rigidity for certain subgroups of
  $\mathrm{Diff}^{\omega}(\mathrm{S}^1)$.
\newblock {\em Ann. Sci. \'Ecole Norm. Sup.}, 32:433--453, 1999.

\bibitem{Sel}
A.~Selberg.
\newblock On discontinuous groups in higher-dimmensional symmetric spaces.
\newblock In {\em Contributions to function theories}, pages 147--164, Bombay,
  1960. Tata Institute of Fundamental Research.

\bibitem{Ser}
F.~Sergeraert.
\newblock Feuilltages et diff\'eomorphismes infiniment tangents \`a
  l'identit\'e.
\newblock {\em Invent. Math.}, 39:253--275, 1977.

\bibitem{Sz}
G.~Szekeres.
\newblock Regular iteration of real and complex functions.
\newblock {\em Acta Math.}, 100:203--258, 1958.

\end{thebibliography}

\end{document}